\documentclass[12pt]{amsart}
\usepackage{amsmath}
\usepackage{graphicx}
\usepackage{hyperref}
\usepackage[latin1]{inputenc}
\usepackage{mathtools}
\usepackage{amsfonts}
\usepackage{amssymb}
\usepackage[T1]{fontenc}
\usepackage{amsthm}
\usepackage{fullpage}
\usepackage{enumitem}

\newtheorem{theorem}{Theorem}[section]

\newtheorem{lemma}[theorem]{Lemma}
\newtheorem{Conjecture}[theorem]{Conjecture}

\theoremstyle{definition}
\newtheorem{definition}{Definition}[section]

\theoremstyle{remark}

\numberwithin{equation}{section}

\newcommand{\F}{\mathbb{F}_q}
\newcommand{\Fm}{\mathbb{F}_{q^m}}
\newcommand{\N}{\mathfrak{N}}

\newcommand{\R}{\mathfrak{R}}
\newcommand{\T}{\mathfrak{T}}
\newcommand{\E}{\mathfrak{E}}
\newcommand{\vt}{\vartheta}
\newcommand{\si}{\sigma}
\newcommand{\Nf}{\mathfrak{N}_{f,a}(e_1, e_2, g_1, g_2)}
\newcommand{\C}{\chi_{f,a}(d_1, d_2, h_1, h_2)}
\DeclareMathOperator{\Rad}{Rad}
\DeclareMathOperator{\Tr}{Tr}

\title{Primitive normal pairs of elements with one prescribed trace}


\author{Arpan Chandra Mazumder}
\address{Department of Mathematical Sciences, Tezpur University, Tezpur, Assam, 784028, India}
\email{arpan10@tezu.ernet.in}

\author{Himangshu Hazarika}
\address{Department of Mathematics, Tezpur College, Tezpur, Assam, 784001, India}
\email{diku\_95@tezu.ernet.in}

\author{Dhiren Kumar Basnet}
\address{Department of Mathematical Sciences, Tezpur University, Tezpur, Assam, 784028, India}
\email{dbasnet@tezu.ernet.in}

\author{Giorgos Kapetanakis}
\address{Department of Mathematics, University of Thessaly, 3rd km Old National Road Lamia-Athens, 35100, Lamia, Greece}
\email{kapetanakis@uth.gr}

\thanks{The first author is supported by DST INSPIRE Fellowship, under grant no. DST/INSPIRE Fellowship/2021/IF210206. }

\begin{document}
	\begin{abstract}
		Let $q, n, m \in \mathbb{N}$ be such that $q$ is a prime power, $m \geq 3$ and $a \in \F$. We establish a sufficient condition for the existence of a primitive normal pair $(\alpha, f(\alpha))$ in $\mathbb{F}_{q^m}$ over $\mathbb{F}_{q}$ such that $\Tr_{\mathbb{F}_{q^m}/\mathbb{F}_{q}}(\alpha^{-1})=a$, where $f(x) \in \mathbb{F}_{q^m}(x)$ is a rational function with degree sum $n$. In particular, for $q=5^k, ~k \geq 3$ and degree sum $n=4$, we explicitly find at most 25 choices of $(q, m)$ where existence of such pairs is not guaranteed.
	\end{abstract}
	
	\maketitle
	
	\section{Introduction}
	
	Let $q$ be a prime power. We denote by $\F$, the finite field of order $q$. For any given positive integer $m$, let $\Fm$ denote the extension field of $\F$ of degree $m$. The multiplicative group $\Fm^*$ is cyclic and a generator of this group is called a primitive element of $\Fm^*$. An element $\alpha \in \Fm$ is said to be normal over $\F$ if the set of all its conjugates with respect to $\F$, that is, if the set $\{\alpha, \alpha^{q}, \ldots , \alpha^{q^{m-1}}\}$  forms a basis of $\Fm$ over $\F$. The trace of an element $\alpha \in \Fm$ over $\F$, denoted by $\Tr_{\mathbb{F}_{q^m}/\mathbb{F}_{q}}(\alpha)$, is the sum of all conjugates of $\alpha$ with respect to $\F$, that is, $\Tr_{\mathbb{F}_{q^m}/\mathbb{F}_{q}}(\alpha)= \alpha+\alpha^{q}+\ldots+\alpha^{q^{m-1}}$. 
	
	An element $\alpha \in \Fm$ is said to be a primitive normal if it is both primitive and normal over $\F$. We refer the readers to \cite{lidl} for the existence of primitive and normal elements in finite fields. In \cite{lens}, Lenstra and Schoof first established the simultaneous existence of both primitive and normal elements in finite fields. Later on, by implementing sieving techniques, Cohen and Huczynska \cite{cohen1} gave a computer-free proof of it. 
	
	For any $\alpha, \beta \in \Fm$, we call a pair $(\alpha, \beta)$ to be a primitive normal pair if both $\alpha$ and $\beta$ are primitive and normal elements over $\F$. For any rational function $f(x) \in \Fm(x)$, the existence of primitive pairs and also primitive normal pairs $(\alpha, f(\alpha))$ has been an area of much research interest and a significant amount of work has been published in this direction \cite{sharma2017existence, sharma2018existence, basnet, avnish}. 
	
	In 1999, Cohen and Hachenberger \cite{cohen2} proved the existence of a primitive normal element $\alpha \in \F$ with a prescribed trace. In 2001, Chou and Cohen \cite{chou} proved the existence of an element $\alpha$ such that both $\alpha$ and $\alpha^{-1}$ have zero traces over $\F$. In 2018, Gupta, Sharma and Cohen \cite{gupta1}, proved the existence of a primitive pair $(\alpha, \alpha+\alpha^{-1})$ with $\Tr_{\mathbb{F}_{q^m}/\mathbb{F}_{q}}(\alpha)=a$, for any prescribed $a \in \F$ and for all $q$ and $n \geq 5$. Recently, R. K. Sharma et al. \cite{hariom} proved the existence of a primitive pair $(\alpha, f(\alpha))$ such that $\alpha$ is normal over $\F$ and $\Tr_{\mathbb{F}_{q^m}/\mathbb{F}_{q}}(\alpha^{-1})=a$, for a prescribed $a \in \F$, where $f(x) \in \mathbb{F}_{q^m}(x)$ is a rational function with some exceptions. 
	
	In this article, we consider a primitive normal element $\alpha \in \Fm$ and 
	explore the existence of a primitive normal pair $(\alpha, f(\alpha))$ such that $\Tr_{\mathbb{F}_{q^m}/\mathbb{F}_{q}}(\alpha^{-1})=a$, where $f(x) \in \mathbb{F}_{q^m}(x)$ is a rational function with some minor restrictions. We consider rational functions $f \in \Fm(x)$ which belong to the set $\E_n$ as defined below.
	
	\begin{definition}
		The set $\E_n$ is the collection of rational functions $f \in \Fm(x)$ with simplest form $f_1/f_2$, where $n_1, n_2$ are the degrees of $f_1, f_2$ respectively, with $n_1+n_2=n$ i.e. degree sum of $f$ is $n$, such that
		
		\begin{enumerate}[label=(\roman*)]
			\item $f \neq cx^ig^d$ for some $c \in \Fm$, $i \in \mathbb{Z}$, $g \in \Fm(x)$ and $d(> 1)$ divides $q^m-1$.
			\item \label{part(ii)} $f_2$ is monic such that $f_2\neq g^{q^m}$ for any $g\in\Fm[x]$.
		\end{enumerate}
	\end{definition}
	
	We note that if $f(x) = g(x)^d$ for some $g \in \Fm(x)$ and $d>1$ which divides $q^m-1$, then $f(\alpha)$ (for any primitive element $\alpha$) is necessarily a $d^\text{th}$ power and hence it cannot be primitive. Also, part \ref{part(ii)} of the above definition forces that $n_2 \geq 1$. For the case $n_2 = 0$, i.e. for polynomials $f \in \Fm[x]$, the existence of a primitive normal element $\alpha \in \Fm$ over $\F$ such that $f(\alpha)$ is also a primitive normal element in $\Fm$ over $\F$ with prescribed $a \in \F$ has been studied in \cite{avnish2}.
	
	Furthermore, from \cite{chou} we have that, for $m \leq 2$, there does not exist any primitive element $\alpha \in \Fm$ such that $\Tr_{\mathbb{F}_{q^m}/\mathbb{F}_{q}}(\alpha^{-1})=0$. Therefore, throughout the article we shall assume that $m \geq 3$. Next we introduce some sets and notations which play an important role in this article. For $q=p^k$, we denote by $\mathbb{F}$ the algebraic closure of $\mathbb{F}_p$. Let $\T$ denote the set of pairs $(q, m)$ such that, for any $a \in \F$ and $f\in\E_n$ there exists an element $\alpha \in \Fm$ with $\Tr_{\mathbb{F}_{q^m}/\mathbb{F}_{q}}(\alpha^{-1})=a$, for which $(\alpha, f(\alpha))$ is  a primitive normal pair  in $\mathbb{F}_{q^m}$ over $\mathbb{F}_{q}$. For each $n \in \mathbb{N}$, we denote by $\omega(n)$ and $W(n)$, the number prime divisors of $n$ and the number of square-free divisors of $n$ respectively. Also for $r(x) \in \F[x]$, we denote by $\omega(r)$ and $W(r)$, the number of monic irreducible $\F$-divisors of $r$ and the number of square-free $\F$-divisors of $r$ respectively.
	
	More precisely, in Theorem~\ref{main_theorem}, we obtain an asymptotic condition for the existence of the desired pair $(\alpha,f(\alpha))$. Next, in Sections~\ref{sieve} and \ref{modified_sieve}, we relax this condition, which enables us, in Theorem~\ref{main_exact_theorem}, to obtain explicit result for fields of characteristic $5$.
	
	All non-trivial computations wherever needed in this article are carried out using SageMath \cite{sagesage}.
	
	\section{Preliminaries} 
	In this section, we recall some definitions and results and provide some preliminary notations which are used to prove the main results of this article.
	
	In our work, the characteristic functions of primitive and normal elements play an important role. To represent those functions, the idea of character of finite abelian group is necessary.
	
	\begin{definition}[Character]
		Let $\mathbb{G}$ be a finite abelian group. A character $\chi$ of the group $\mathbb{G}$ is a homomorphism from $\mathbb{G}$ to $\mathbb{S}^1$, where $\mathbb{S}^1:= \lbrace z\in \mathbb{C}: |z| = 1 \rbrace$ is the multiplicative group of complex numbers of unit modulus, i.e. $\chi(a_1a_2)= \chi(a_1)\chi(a_2)$, for all $a_1, a_2 \in \mathbb{G}$.
	\end{definition}
	
	The characters of $\mathbb{G}$ form a group under multiplication defined by $(\chi_\alpha\chi_\beta)(a)= \chi_\alpha(a)\chi_\beta(a)$, called dual group or character group of $\mathbb{G}$ denoted by $\widehat{\mathbb{G}}$, which is isomorphic to the group $\mathbb{G}$. If the order of an element (i.e. a character) of the group $\widehat{\mathbb{G}}$ is $d$, then  characters of order $d$ is denoted by $(d)$. Further, since $\widehat{\mathbb{F}^*_{q^m}} \cong \mathbb{F}^*_{q^m}$, $\widehat{\mathbb{F}^*_{q^m}}$ is cyclic and for any divisor $d$ of $q^m-1$ there are exactly $\phi(d)$ characters of order $d$ in $\widehat{\mathbb{F}^*_{q^m}}$. The special character $\chi_0: \mathbb{G} \rightarrow \mathbb{S}^1$ defined as  $\chi_0(a) = 1$, for all $a \in \mathbb{G}$ is called the trivial character of $\mathbb{G}$.
	
	In a finite field $\mathbb{F}_{q^m}$ there are two group structures, one is the additive group $\mathbb{F}_{q^m}$ and another is the multiplicative group $\mathbb{F}^*_{q^m}$. Therefore we have two types of characters pertaining to these two group structures, one is the additive character for $\mathbb{F}_{q^m}$ denoted by $\psi$ and the another one is the multiplicative character for $\mathbb{F}^*_{q^m}$ denoted by $\chi$. The multiplicative characters associated with $\mathbb{F}^*_{q^m}$ are extended from $\mathbb{F}^*_{q^m}$ to $\mathbb{F}_{q^m}$ by the rule
	\[ \chi(0)=\begin{cases}
		0,& \text{if } \chi\neq\chi_0,\\
		1,& \text{if } \chi=\chi_0. 
	\end{cases} \]
	\begin{definition}[$e$-free element]
		Let $e\mid q^m-1$, then an element $\alpha \in \mathbb{F}^{*}_{q^m}$ is called $e$-free, if $d\mid e$ and $\alpha = y^d$, for some $y \in \mathbb{F}^{*}_{q^m}$ imply $d=1$.
	\end{definition}
	It is clear from the definition that an element $\alpha \in \mathbb{F}^{*}_{q^m}$ is primitive if and only if it is a $(q^m-1)$-free element.  The $p$-free part of a natural number $r$ is denoted by $r_0$, where $r= r_0 p^k$ such that $\mathrm{gcd}(r_0, p) = 1$. It is clear from the definition that $q^m-1$ can be freely replaced by its $p$-free part.
	
	For any $e\mid q^m-1$, the characteristic function of $e$-free elements of $\mathbb{F}^{*}_{q^m}$ is defined as follows:
	\begin{equation}\label{e-free ch}
		\rho_e: \mathbb{F}^{*}_{q^m}\rightarrow \{0,1\}; \alpha \mapsto \theta(e) \sum_{d\mid e} \left( \frac{\mu(d)}{\phi(d)} \sum_{(d)} \chi_d(\alpha) \right),
	\end{equation}
	where $\theta(e):= \phi(e)/e$, $\mu$ is the M\"obius function and $\chi_d$ stands for any character of $\widehat{\mathbb{F}_{q^m}^*}$ of order $d$.
	
	The additive group of $\mathbb{F}_{q^m}$ is an $\mathbb{F}_{q}[x]$-module under the rule \par \hspace{3cm}$f\circ\alpha=\sum_{i=0}^n a_i\alpha^{q^i}$; for $\alpha\in \mathbb{F}_{q^m}$ and $f(x)= \sum_{i=0}^na_ix^i\thinspace \in \mathbb{F}_{q}[x]$.\\ For $\alpha \in \mathbb{F}_{q^m}$, the $\mathbb{F}_q$-order of $\alpha$ is the monic $\mathbb{F}_q$-divisor $g$ of $x^m-1$ of minimal degree such that $g \circ\alpha=0$.
	
	The \emph{$\mathbb{F}_q$-order} of an additive character $\chi_g \in \widehat{\mathbb{F}_{q^m}}$ is the monic $\mathbb{F}_{q}$-divisor $g$ of $x^m-1$ of minimal degree such that $\chi_g \circ g$ is the trivial character of $\widehat{\mathbb{F}_{q^m}}$, where ($\chi_g \circ g)(\alpha)= \chi_g(g \circ \alpha)$ for any $\alpha \in \Fm$.
	
	\begin{definition}[$g$-free element]
		For $g\mid x^m-1$, an element $\alpha\in \mathbb{F}_{q^m}$ is called $g$-free element if $\alpha = h \circ \beta$ for some $\beta \in \mathbb{F}_{q^m}$ and $h\mid g$ imply $h=1$.
	\end{definition}
	
	From the definition, it is obvious that an element $\alpha \in \mathbb{F}_{q^m}$ is normal if and only if it is $(x^m-1)$-free. It is clear that $x^m-1$ can be freely replaced by its $p$-radical $g_0:= x^{m_0}-1$, where $m_0$ is such that $m=m_0p^a$, here $a$ is a non-negative integer and $\mathrm{gcd} (m_0, p)=1$. 
	
	For any $g\mid x^m-1$, the characteristic function of $g$-free elements of $\mathbb{F}_{q^m}$ is defined as follows:
	\begin{equation}\label{g-free ch}
		\eta_g : \mathbb{F}_{q^m}\rightarrow \{0, 1\}; \alpha \mapsto \Theta(g) \sum_{f\mid g} \left( \frac{\mu^\prime(f)}{\Phi(f)} \sum_{(f)} \psi_f(\alpha) \right),
	\end{equation}
	where $\Theta(g):= \Phi(g)/{q^{deg(g)}}$, $\psi_f$ stands for any character of $\widehat{\mathbb{F}_{q^m}}$ of $\mathbb{F}_{q}$-order $f$ and  $\mu^\prime$ is the analogue of the M\"obius function defined as follows
	\[
		\mu^\prime(g)=\begin{cases}
			(-1)^s , & \text{if $g$ is the product of $s$ distinct irreducible monic polynomials}, \\
			0 , &\text{otherwise.}\\ 
		\end{cases}
	\]
	
	For each $a \in \mathbb{F}_{q}$ the characteristic function of the subset of $\mathbb{F}_{q^m}$ consisting of elements $\alpha$ with $\Tr_{\mathbb{F}_{q^m}/\mathbb{F}_{q}}(\alpha)=a$ is defined as follows: 
	\begin{equation}\label{Tr ch}
		\tau_a : \mathbb{F}_{q^m}\rightarrow \{0, 1\}; \alpha \mapsto \frac{1}{q}\sum_{\psi \in \widehat{\mathbb{F}_{q}}} \psi\left(\text{Tr}_{\mathbb{F}_{q^m}/\mathbb{F}_{q}}(\alpha)-a \right).
	\end{equation}
	
	The additive character $\psi_1$ defined by $\psi_1(\beta) = e^{{2{\pi}i\text{Tr}(\beta)}/p}$, for all $\beta \in \mathbb{F}_{q}$, where Tr is the absolute trace function from $\mathbb{F}_{q}$ to $\mathbb{F}_{p}$,
	is called the \emph{canonical additive character} of $\mathbb{F}_{q}$ and every additive character $\psi_\beta$
	for $\beta \in \mathbb{F}_{q}$ can be expressed in terms of the canonical additive character $\psi_1$ as
	$\psi_\beta(\gamma)=\psi_1(\beta\gamma)$, for all $\gamma \in \mathbb{F}_{q}$.
	
	Thus we have that 
	\begin{align*}
		\tau_a(\alpha) & =  \frac{1}{q}\sum_{u \in \mathbb{F}_{q}} \psi_1\left(\text{Tr}_{\mathbb{F}_{q^m}/\mathbb{F}_{q}}(u\alpha)-ua \right) \\ 
		&=  \frac{1}{q}\sum_{u \in \mathbb{F}_{q}} \hat{\psi_1}(u\alpha)\psi_1(-ua),
	\end{align*}
	where $\hat{\psi_1}$ is the lift of $\psi_1$, that is, the additive character of $\mathbb{F}_{q^m}$ defined by $\hat{\psi_1}(\alpha) = \psi_1(\text{Tr}_{{\mathbb{F}_{q^m}}/\mathbb{F}_{q}}(\alpha))$. In particular, $\hat{\psi_1}$ is the canonical additive character of $\mathbb{F}_{q^m}$.
	
	The following result was given by Weil \cite{weil} as described in \cite{pinner} at (1.1) to (1.3).
	
	\begin{lemma}\label{2.1}
		Let $f(x) \in \Fm(x)$ be a rational function. Write $f(x)= \prod_{j=1}^k f_j(x)^{r_j}$, where $f_j(x) \in \mathbb{F}_{q^m}[x]$ are irreducible polynomials and $r_j$ are non zero integers. Let $\chi$ be a non-trivial multiplicative character of $\mathbb{F}_{q^m}$ of square-free order $d$ (a divisor of $q^m-1$). Suppose that $f(x)$ is not of the form $cg(x)^d$ for any rational function $g(x)$ in $\mathbb{F}(x)$, where $\mathbb{F}$ is the algebraic closure of $\mathbb{F}_{q^m}$ and $c \in \mathbb{F}^*_{q^m}$. Then we have
		\begin{equation*}
			\left| \sum_{\alpha \in \mathbb{F}_{q^m},f(\alpha)\neq 0, \infty} \chi(f(\alpha)) \right| \leq \left({\sum_{j=1}^{k} deg(f_j)}-1\right)q^{\frac{m}{2}}.
		\end{equation*}
	\end{lemma}
	
	\begin{lemma}\cite{castro}\label{2.2}
		Let $\chi$ be a non-trivial multiplicative character of order $r$ and $\psi$ be a non-trivial additive character of $\Fm$. Let $f$, $g$ be rational functions in $\Fm(x)$ such that $f \neq yh^r$, for any $y \in \Fm$, $h \in \Fm(x)$, and $g \neq h^p-h+y$ for any $y \in \Fm$, $h \in \Fm(x)$. Then
		
		\begin{equation*}
			\left|\sum_{x\in \Fm\setminus S}\chi(f(x))\psi(g(x))\right| \leq ((\deg g)_{\infty}+l+l'-l''-2)q^{m/2},
		\end{equation*}
		
		where $S$ is the set of poles of $f$ and $g$, $(g)_{\infty}$ is the pole divisor of $g$, $l$ is the number of distinct zeros and finite poles of $f$ in $\mathbb{F}$(algebraic closure of $\F$), $l'$ is the number of distinct poles of $g$ (including $\infty$) and $l''$ is the number of finite poles of $f$ that are poles or zeros of $g$.
	\end{lemma}

	\section{The Main Result}
	
	In this section, we present a sufficient condition for the existence of elements $\alpha\in \Fm$ such that both  $\alpha$ and $f(\alpha)$ are simultaneously primitive normal elements in $\mathbb{F}_{q^m}$ over $\F$ such that $\Tr_{\mathbb{F}_{q^m}/\mathbb{F}_{q}}(\alpha^{-1})=a$, where $f(x)\in\E_n$. 
	
	Let $e_1,e_2$ be such that $e_1, e_2\mid q^m-1$ and $g_1,g_2$ be such that $g_1,g_2\mid x^m-1$. Denote by $\mathfrak{N}_{f,a}(e_1,e_2,g_1,g_2)$, the number of $\alpha\in\mathbb{F}_{q^m}$, such that $\alpha$ is both $e_1$-free and $g_1$-free and $f(\alpha)$ is  $e_2$-free and $g_2$-free with $\Tr_{\mathbb{F}_{q^m}/\mathbb{F}_{q}}(\alpha^{-1})=a$; where $f(x) \in \E_n$. Let $P$ be the set containing $0$, and zeros and poles of $f$.
	
	\begin{theorem} \label{main_theorem}
		Let $q$ be a prime power and $m, n \in \mathbb{N}$ such that $m \geq 3$. Suppose $e_1, e_2$ divide $q^m-1$ and $g_1,g_2$ divide $x^m-1$. Then we have that 
		\begin{equation*}
			\mathfrak{N}_{f,a}(e_1,e_2,g_1,g_2) \geq \theta(e_1)\theta(e_2)\Theta(g_1)\Theta(g_2)q^{m/2}(q^{m/2}-(2n+1)W(e_1)W(e_2)W(g_1)W(g_2)).
		\end{equation*}
		In particular, we have $(q, m) \in \T$ provided that 
		\begin{equation}\label{cond3.1}
			q^{m/2-1} > (2n+1)W(q^m-1)^2W(x^m-1)^2.
		\end{equation}
	\end{theorem}
	
	\begin{proof}
		To prove this result, it is enough to show that $\Nf > 0$ for every $f \in \E_n$ and prescribed $a \in \F$. From the definition we have, 
		\begin{equation*}
			\Nf = \sum_{\alpha\in\mathbb{F}_{q^m} \setminus P}\rho_{e_1}(\alpha)\rho_{e_2}(f(\alpha))\eta_{g_1}(\alpha)\eta_{g_2}(f(\alpha))\tau_{a}(\alpha^{-1}).
		\end{equation*}
		Now using (\ref{e-free ch}), (\ref{g-free ch}) and (\ref{Tr ch}), we have 
		\begin{align}\label{cond3.2}
			\Nf =& ~\frac{\theta(e_1)\theta(e_2)\Theta(g_1)\Theta(g_2)}{q}
			\sum_{\substack{d_1\mid e_1, d_2\mid e_2 \\ h_1\mid g_1, h_2\mid g_2}} \frac{\mu(d_1)\mu(d_2)\mu'(h_1)\mu'(h_2)}{\phi(d_1)\phi(d_2)\Phi(h_1)\Phi(h_2)}
			\sum_{\substack{\chi_{d_1}, \chi_{d_2} \\ \psi_{h_1}, \psi_{h_2}}}
			\chi_{f,a}(d_1, d_2, h_1, h_2) \nonumber\\
			=& ~\vt\sum_{\substack{d_1\mid e_1,d_2\mid e_2 \\ h_1\mid g_1,h_2\mid g_2}} \frac{\mu(d_1)\mu(d_2)\mu'(h_1)\mu'(h_2)}{\phi(d_1)\phi(d_2)\Phi(h_1)\Phi(h_2)} \sum_{\substack{\chi_{d_1},\chi_{d_2} \\ \psi_{h_1},\psi_{h_2}}} \chi_{f,a}(d_1, d_2, h_1, h_2),
		\end{align}
		where $\vt = \theta(e_1)\theta(e_2)\Theta(g_1)\Theta(g_2)/{q}$ and 
		\begin{equation*}
			\C = \sum_{u\in\F} \psi_1(-au)\sum_{\alpha\in\Fm\setminus P} \chi_{d_1}(\alpha)\chi_{d_2}(f(\alpha))\psi_{h_1}(\alpha)\psi_{h_2}(f(\alpha))\hat{\psi}_1(u\alpha^{-1}).
		\end{equation*}
		Further, there exist $u_1, u_2 \in \Fm$ such that $\psi_{h_i}(\alpha)=\widehat{\psi}_1(u_i\alpha)$, for $i = 1, 2$, where $\hat{\psi}_1$ is the canonical additive character of $\Fm$. Hence 
		\begin{equation}\label{cond3.3}
			\C = \sum_{u\in\F} \psi_1(-au)\sum_{\alpha\in\Fm\setminus P} \chi_{d_1}(\alpha)\chi_{d_2}(f(\alpha))\hat{\psi}_1(u_1\alpha+u_2f(\alpha)+u\alpha^{-1}).
		\end{equation}
		Now, for $d_2=1$ we have that 
		
		\begin{equation*}
			\C = \sum_{u\in\F} \psi_1(-au)\sum_{\alpha\in\Fm\setminus P} \chi_{d_1}(\alpha)\hat{\psi}_1(u_1\alpha+u_2f(\alpha)+u\alpha^{-1}) .
		\end{equation*}
		If $u_1x+u_2f(x)+ux^{-1} \neq \R(x)^{p}-\R(x)+y$ for any $\R(x) \in \mathbb{F}(x)$ and $y \in \Fm$, then using Lemma \ref{2.2}, we get that
		\begin{align*}
			|\C| \leq& (2n+1)q^{m/2+1}\\
		\end{align*}
		Suppose 
		\begin{equation} \label{zero_us}
		u_1x+u_2f(x)+ux^{-1} = \R(x)^{p}-\R(x)+y,
		\end{equation}
		for some $\R(x) \in \mathbb{F}(x)$ and $y \in \Fm$. We will show that this is feasible only if $u_1=u_2=u=0$. Write $f={f_1}/{f_2}$ and $\R={r_1}/{r_2}$, where $f_1, f_2$ are co-prime polynomials over $\Fm$, $r_1, r_2$ are co-prime polynomials over $\mathbb{F}$ and $f_2, r_2$ are both monic. Then we have that 
		\begin{equation}\label{cond3.4}
			r_2^{q^m}(u_1x^2f_2+u_2xf_1+uf_2)=xf_2(r_1^p-r_1r_2^{p-1}+yr_2^p) .
		\end{equation}
     
	Suppose $u_2\neq 0$.
	Since $\gcd(r_1, r_2)=1$, we get that $r_2^{q^m}\mid xf_2$. 
	Further, since $\gcd(f_1,f_2)=1$, we get that $f_2\mid r_2^{q^m} \Rightarrow r_2^{q^m} = \kappa f_2$. It follows that, $r_2^{q^m} = f_2$ or $r_2^{q^m} = xf_2$. The former case implies that $f\not\in\T_n$, a contradiction, thus, 
	\begin{equation} \label{r2=xf2}
	r_2^{q^m} = xf_2.
	\end{equation}
	By combining, \eqref{cond3.4} and \eqref{r2=xf2}, we obtain
	\begin{equation} \label{cond3.4_2}
	u_1x^2f_2+u_2xf_1+uf_2=r_1^p-r_1r_2^{p-1}+yr_2^p.
	\end{equation}
	However, notice that \eqref{r2=xf2} implies that $x\mid f_2$ and $x\mid r_2$ and, plugging in the above to \eqref{cond3.4_2}, one obtains $x\mid r_1$, which contradicts  $\gcd(r_1,r_2)=1$. It follows that $u_2=0$. Now, \eqref{zero_us} becomes $u_1x+ux^{-1} = \R(x)^{p}-\R(x)+y$, which can easily lead to a contradiction if we further assume that $u_1\neq 0$ or $u\neq 0$, by comparing degrees.

	Therefore, when $u_1x+u_2f(x)+ux^{-1} = \R(x)^{p}-\R(x)+y$ for $\R(x) \in \mathbb{F}(x)$ and $y \in \Fm$, then $u_1=u_2=u=0$. This implies 
	\begin{equation*}
			|\C| \leq |P|q \leq (2n+1)q^{m/2+1}.
	\end{equation*}
	Next suppose $d_2 >1$. Let $d$ be the least common multiple of  $d_1$ and $d_2$. Then there exists a character $\chi_d$ of order $d$ such that $\chi_{d_2}=\chi_d^{d/d_2}$. Again there is an integer $0 \leq k \leq q^m-1$ such that $\chi_{d_1}=\chi_d^k$. Substituting these in (\ref{cond3.3}), we have the following expression,
	\begin{equation}\label{cond3.5}
			\C = \sum_{u\in\F} \psi_1(-au)\sum_{\alpha\in\Fm\setminus P} \chi_{d}(\alpha^kf(\alpha)^{d/d_2})\hat{\psi}_1(u_1\alpha+u_2f(\alpha)+u\alpha^{-1}) .
	\end{equation}
		
	Now let us first suppose that the inner sum is non-degenerate. Then Lemma \ref{2.2} gives
		\begin{equation*}
			|\C| \leq (2n+1)q^{m/2+1} .
		\end{equation*} 
		Again if $u_1x+u_2f(x)+ux^{-1} = \R(x)^{q^m}-\R(x)$ for some $\R(x) \in \mathbb{F}(x)$, then $u_1=u_2=u=0$. Substituting $u_1=u_2=u=0$ in (\ref{cond3.5}) we get
		\begin{equation}\label{cond3.6}
			\C = \sum_{u\in\F} \psi_0(-au)\sum_{\alpha\in\Fm\setminus P} \chi_{d}(\alpha^kf(\alpha)^{d/d_2}).
		\end{equation}
		In this case if $x^kf(x)^{d/d_2} \neq c\R(x)^d$ for any polynomial $\R(x) \in \Fm[x]$ and $c \in \Fm^*$, then Lemma \ref{2.1} implies that 
		\begin{equation*}
			|\C| \leq nq^{m/2+1} \leq (2n+1)q^{m/2+1} .
		\end{equation*}
		Next suppose that $x^kf(x)^{d/d_2} = c\R(x)^d$ for some polynomial $\R(x) \in \Fm(x)$ and $c \in \Fm^*$. Then we have that $f \notin \E_n$. 
		
		Summarizing all the above cases we observe that 
		\begin{equation*}
			|\C| \leq (2n+1)q^{m/2+1},
		\end{equation*}
		when $(d_1,d_2,h_1,h_2) \neq (1, 1, 1, 1)$. 
		
		Therefore  from (\ref{cond3.2}), we get  
		\begin{align}\label{cond3.7}
			\Nf \geq& ~\vt[q^m-(n+1)q-((2n+1)q^{m/2+1})(W(e_1)W(e_2)W(g_1)W(g_2)-1)] \nonumber \\ >& ~\vt[q^m-(2n+1)q^{m/2+1}W(e_1)W(e_2)W(g_1)W(g_2)] .
		\end{align}
		Hence $\Nf>0$ if $q^{m/2}>(2n+1)W(e_1)W(e_2)W(g_1)W(g_2)$.
		
		It is clear from the above discussion that $(q,m) \in \T$, when $e_1=e_2=q^m-1$ and $g_1=g_2=x^m-1$, that is provided 
		\[
			q^{m/2-1} > (2n+1)W(q^m-1)^2W(x^m-1)^2. \qedhere
		\]
	\end{proof}

	\section{The prime sieve technique}\label{sieve}
	
	In this section we apply the sieving inequality established by Kapetanakis in \cite{kapetanakis2014normal} to give a modified form of (\ref{cond3.7}).
	
	\begin{lemma}[Sieving Inequality] \label{sieveineq}
		
		Let $d$ be a divisor of $q^m-1$ and $p_1, p_2,\dots , p_r$ be the remaining distinct primes dividing $q^m-1$. Furthermore, let $g$ be a divisor of $x^m-1$ such that $g_1, g_2, \dots , g_s$ are the remaining distinct irreducible factors of $x^m-1$. Abbreviate  $ \mathfrak{N}_{f,a}(q^m-1, q^m-1, x^m-1, x^m-1) $ to $\mathfrak{N}_{f,a}$. Then
		\begin{multline} \label{cond4.1}
			\mathfrak{N}_{f,a} \geq \sum_{i=1}^r\mathfrak{N}_{f,a}(p_i d, d, g, g)+  \sum_{i=1}^r\mathfrak{N}_{f,a}( d,p_i d, g, g)+ \sum_{i=1}^s\mathfrak{N}_{f,a}(d, d,g_i g, g)\\ 
			+  \sum_{i=1}^s\mathfrak{N}_{f,a}(d, d, g,g_i g) -(2r+2s-1)\mathfrak{N}_{f,a}(d, d, g, g).
		\end{multline}
	\end{lemma}
	
	An upper bound for some differences that will be needed later is provided by the following result.
	
	\begin{lemma}\label{diffub}
		Let $d, m, q \in \mathbb{N}$, $g \in \F[x]$ be such that q is a prime power, $m \geq 3$, $d \mid  q^m-1$ and $g \mid  x^m-1$. Let $e$ be a prime number which divides $q^m-1$ but not $d$, and $h$ be an irreducible polynomial dividing $x^m-1$ but not $g$. Then we have the following bounds:
		\begin{align*}
		|\N_{f,a}(ed, d, g, g)-\theta(e)\N_{f,a}(d, d, g, g)| & \leq (2n+1)\theta(e)\theta(d)^2\Theta(g)^2W(d)^2W(g)^2q^{m/2}, \\
		|\N_{f,a}(d, ed, g, g)-\theta(e)\N_{f,a}(d, d, g, g)| & \leq (2n+1)\theta(e)\theta(d)^2\Theta(g)^2W(d)^2W(g)^2q^{m/2},\\
		|\N_{f,a}(d, d, hg, g)-\Theta(h)\N_{f,a}(d, d, g, g)| & \leq (2n+1)\Theta(h)\theta(d)^2\Theta(g)^2W(d)^2W(g)^2q^{m/2},\\
		|\N_{f,a}(d, d, g, hg)-\Theta(h)\N_{f,a}(d, d, g, g)| & \leq (2n+1)\Theta(h)\theta(d)^2\Theta(g)^2W(d)^2W(g)^2q^{m/2} .
		\end{align*}
	\end{lemma}
	
	\begin{proof}
		From the definition, we have 
		\begin{multline*}
			\N_{f,a}(ed, d, g, g)-\theta(e)\N_{f,a}(d, d, g, g) \\= \vt
			\sum_{\substack{e\mid d_1\mid ed ,\ d_2\mid d \\ g_1\mid g , \ g_2\mid g}}
			\frac{\mu(d_1)\mu(d_2)\mu'(g_1)\mu'(g_2)}{\phi(d_1)\phi(d_2)\Phi(g_1)\Phi(g_2)}\sum_{\substack{\chi_{g_1},\chi_{g_2} \\ \psi_{h_1},\psi_{h_2}}} \chi_{f,a}(d_1, d_2, g_1, g_2) .
		\end{multline*}
		Now, using $|\chi_{f,a}(d_1, d_2, g_1, g_2)| \leq (2n+1)q^{m/2+1}$, we get 
		\begin{multline*}
			|\N_{f,a}(ed, d, g, g)-\theta(e)\N_{f,a}(d, d, g, g)| \\ \leq \frac{\theta(e)\theta(d)^2\Theta(g)^2}{q}(2n+1)q^{m/2+1}W(d)\left(W(ed)-W(d)\right)W(g)^2 .
		\end{multline*}
		Since $W(ed)=W(e)W(d)=2W(d)$, we have 
		\begin{equation*}
			|\N_{f,a}(ed, d, g, g)-\theta(e)\N_{f,a}(d, d, g, g)| \leq \frac{\theta(e)\theta(d)^2\Theta(g)^2}{q}(2n+1)q^{m/2+1}W(d)^2W(g)^2 .
		\end{equation*}
		We obtain the remaining bounds in a similar fashion.
	\end{proof}
	
	\begin{theorem}\label{thsieve}
		Let $d, m, q \in \mathbb{N}$, $g \in \F[x]$ be such that q is a prime power, $m \geq 3$, $d \mid  q^m-1$ and $g \mid  x^m-1$. Let $d$ be a divisor of $q^m-1$ and $p_1, p_2,\ldots , p_r$ be the remaining distinct primes dividing $q^m-1$. Furthermore, let $g$ be a divisor of $x^m-1$ such that $g_1, g_2, \dots , g_s$ are the remaining distinct irreducible factors of $x^m-1$. Define 
		\[ \lambda := 1 - 2  \sum_{i=1}^r\frac{1}{p_i} - 2  \sum_{i=1}^s \frac{1}{q^{{\mathrm deg}(g_i)}}, \  \lambda>0 \] and \[ \Lambda  := \frac{2r+2s-1}{\lambda}+2. \]
		Then $\mathfrak{N}_{f,a}>0$ if 
	  \begin{equation}\label{cond4.2}
			q^{m/2-1}>(2n+1)W(l)^2W(g)^2\Lambda.
		\end{equation}
	\end{theorem}
	
	\begin{proof}
		Using Lemma \ref{sieveineq} and Lemma \ref{diffub}, we get that
		\begin{align*}
			\mathfrak{N}_{f,a} \geq& 
			\sum_{i=1}^{r}\left\{\N_{f,a}(p_id, d, g, g)-\theta(p_i)\N_{f,a}(d, d, g, g)\right\}  + \\ & 
			\sum_{i=1}^{r}\left\{\N_{f,a}(d, p_id, g, g)-\theta(p_i)\N_{f,a}(d, d, g, g)\right\}  + \\ &
			\sum_{i=1}^{s}\left\{\N_{f,a}(d, d, g_ig, g)-\theta(g_i)\N_{f,a}(d, d, g, g)\right\}  + \\ &
			\sum_{i=1}^{s}\left\{\N_{f,a}(d, d, g, g_ig)-\theta(g_i)\N_{f,a}(d, d, g, g)\right\}  + 
			\lambda\N_{f,a}(d, d, g, g) .
		\end{align*}
		
		Now using Lemma \ref{diffub}, we have the following expression
		\begin{multline*}
			\mathfrak{N}_{f,a} \geq
			\theta(d)^2\Theta(g)^2\Bigg[\left(2  \sum_{i=1}^{r}\theta(p_i)+2  \sum_{i=1}^{s}\Theta(g_i)\right)\left(-(2n+1)W(d)^2W(g)^2q^{m/2}\right) \\ + \lambda\left\{q^{m-1}-(n+1)-(2n+1)q^{m/2}(W(d)^2W(g)^2-1)\right\}\Bigg] ,
		\end{multline*}
		thus,
		\begin{align*}
			\mathfrak{N}_{f,a} \geq&
			\theta(d)^2\Theta(g)^2\lambda\Bigg[ \left(\frac{2  \sum\limits_{i=1}^{r}\theta(p_i)+2  \sum\limits_{i=1}^{s}\Theta(g_i)}{\lambda}+1\right)\left(-(2n+1)W(d)^2W(g)^2q^{m/2}\right) \nonumber\\ +& \left\{q^{m-1}-(n+1)-(2n+1)q^{m/2}\right\}\Bigg]
		\end{align*}

		We note that $\lambda = 2\sum\limits_{i=1}^{r}\theta(p_i)+2  \sum\limits_{i=1}^{s}\Theta(g_i)-(2r+2s+1)$. Then the above becomes
		\begin{align*}
			\mathfrak{N}_{f,a} \geq&
			~\theta(d)^2\Theta(g)^2\lambda \left[\left(-(2n+1)W(d)^2W(g)^2q^{m/2}\Lambda\right) + \left\{q^{m-1}-(n+1)-(2n+1)q^{m/2}\right\}\right]  , 
		\end{align*}
		where $\Lambda = \frac{2r+2s-1}{\lambda}+2$.
		
		Therefore when $\lambda>0$, \eqref{cond4.2} implies $\mathfrak{N}_{f,a}>0$.
	\end{proof}
		The technique of applying the Prime Sieve criterion given by Theorem \ref{thsieve} is to include all the small primes in $q^m-1$ in $d$ keeping as many as possible in $p_1, p_2, \dots, p_r$ so long as $\lambda$ is positive.
	
	\section{The Modified Prime Sieve Technique}\label{modified_sieve}
	
	In this section we present a modified sieving criterion which is more effective than Theorem~\ref{thsieve}, that was first introduced in \cite{cohen2021primitive}. To proceed further we need the following notations and conventions which we adopt from \cite{cohen2021primitive}. Let $\Rad(q^m-1)=kPL$, where $k$ is the product of the small prime divisors of $q^m-1$, $L$ is the product of the large prime divisors of $q^m-1$ denoted by $L=l_1.l_2\ldots l_t$, and the rest of the prime divisors of $q^m-1$ are in $P$ and denoted by $p_1, p_2,\dots , p_r$. Similarly we have $\Rad(x^m-1)=gGH$, where $g$ is the product of irreducible factors of $x^m-1$ of least degree. Irreducible factors of large degree are factors of $H$ denoted by $h_1,h_2,\ldots,h_u$ and the rest of the factors of $x^m-1$ lie in $G$ and are denoted by $g_1, g_2, \dots , g_s$.
	
	\begin{theorem}\label{modps}
		Let $m,q \in \mathbb{N}$ such that $q$ is a prime power and $m \geq 3$. With notations as above, let $\Rad(q^m-1)=kPL$, $\Rad(x^m-1)=gGH$. Define
		\begin{equation*}
			\lambda := 1 - 2  \sum_{i=1}^r\frac{1}{p_i} - 2  \sum_{i=1}^s \frac{1}{q^{{\mathrm deg}(g_i)}}, \lambda>0,  \epsilon_1:=\sum_{i=1}^t\frac{1}{l_i}, \epsilon_2:=\sum_{i=1}^u\frac{1}{q^{{\mathrm deg}(h_i)}}
		\end{equation*} and let
		$$\lambda\theta(k)^2\Theta(g)^2-(2\epsilon_1+2\epsilon_2)>0.$$ Then  
		\begin{multline}\label{cond5.1}
			q^{\frac m2-1}  > (2n+1)\bigg[\theta(k)^2\Theta(g)^2W(k)^2W(g)^2(2r+2s-1+2\lambda) \\ +\frac{(n+2)(t-\epsilon_1)+(n+3)(u-\epsilon_2)}{(2n+1)} - \frac{n}{(2n+1)q^{m/2}}(t+u-(2\epsilon_1+2\epsilon_2))\bigg] \\ /\bigg[\lambda\theta(k)^2\Theta(g)^2-(2\epsilon_1+2\epsilon_2)\bigg] .
		\end{multline}
		
	\end{theorem}
	
	\begin{proof}
		It is clear that 
		\begin{multline}\label{cond5.2}
			\mathfrak{N}_{f,a}(q^m-1, q^m-1, x^m-1, x^m-1)= \mathfrak{N}_{f,a}(kPL, kPL, gGH, gGH)  \\ \geq \mathfrak{N}_{f,a}(kP, kP, gG, gG) + \mathfrak{N}_{f,a}(L, L, H, H)- \mathfrak{N}_{f,a}(1, 1, 1, 1).
		\end{multline}
		Now using Lemma \ref{sieveineq} we have,
		\begin{multline*}
			\mathfrak{N}_{f, a}(kP, kP, gG, gG) \geq \sum_{i=1}^r\left\{\N_{f,a}(kp_i, k, g, g)-\theta(p_i)\N_{f,a}(k, k, g, g)\right\} \\ + 
			\sum_{i=1}^r\left\{\N_{f,a}(k, kp_i, g, g)-\theta(p_i)\N_{f,a}(k, k, g, g)\right\} \\ +
			\sum_{i=1}^s\left\{\N_{f,a}(k, k, gg_i, g)-\theta(g_i)\N_{f,a}(k, k, g, g)\right\} \\ +
			\sum_{i=1}^s\left\{\N_{f,a}(k, k, g, gg_i)-\theta(g_i)\N_{f,a}(k, k, g, g)\right\} + \lambda\N_{f,a}(k, k, g, g).
		\end{multline*}
		Further, by Lemma \ref{diffub} and (\ref{cond3.7}) 
		\begin{align}\label{cond5.3}
			\mathfrak{N}_{f, a}(kP, kP, gG, gG) \geq& ~\lambda\theta(k)^2\Theta(g)^2\left[q^{m-1}-(2n+1)q^{m/2}W(k)^2W(g)^2\right] \nonumber  \\ & - (2n+1)\theta(k)^2\Theta(g)^2W(k)^2W(g)^2q^{m/2}\left(\sum_{i=1}^r\theta(p_i)+\sum_{i=1}^s\theta(g_i)\right) \nonumber \\
			\geq& ~\theta(k)^2\Theta(g)^2\left[\lambda q^{m-1}-q^{m/2}(2n+1)W(k)^2W(g)^2\left(2r+2s-1+2\lambda\right)\right]
		\end{align}
		Now using Lemma \ref{sieveineq}, we arrive at 
		\begin{multline}
			\mathfrak{N}_{f, a}(L, L, H, H)-\mathfrak{N}_{f, a}(1, 1, 1, 1) \geq  \sum_{i=1}^t\mathfrak{N}_{f, a}(l_i, 1, 1, 1) + \sum_{i=1}^t\mathfrak{N}_{f, a}(1, l_i, 1, 1)+ \nonumber \\   \sum_{i=1}^u\mathfrak{N}_{f, a}(1, 1, h_i, 1)+\sum_{i=1}^u\mathfrak{N}_{f, a}(1, 1, 1, h_i)- (2t+2u)\mathfrak{N}_{f, a}(1, 1, 1, 1) \nonumber \\ 
		\end{multline}
		
		Thus 
		\begin{multline}\label{cond5.4}
			\mathfrak{N}_{f, a}(L, L, H, H)-\mathfrak{N}_{f, a}(1, 1, 1, 1) \geq \sum_{i=1}^t\left\{\mathfrak{N}_{f, a}(l_i, 1, 1, 1)-\theta(l_i)\mathfrak{N}_{f, a}(1, 1, 1, 1)\right\} \\ + \sum_{i=1}^t\left\{\mathfrak{N}_{f, a}(1, l_i, 1, 1)-\theta(l_i)\mathfrak{N}_{f, a}(1, 1, 1, 1)\right\} 
			+  \sum_{i=1}^u\left\{\mathfrak{N}_{f, a}(1, 1, h_i, 1)-\Theta(h_i)\mathfrak{N}_{f, a}(1, 1, 1, 1)\right\} \\ +  \sum_{i=1}^u\left\{\mathfrak{N}_{f, a}(1, 1, 1, h_i) - \Theta(h_i)\mathfrak{N}_{f, a}(1, 1, 1, 1)\right\}  - (2t+2u)\mathfrak{N}_{f, a}(1, 1, 1, 1).
		\end{multline}
		
		Let $l$ be a prime divisor of $q^m-1$. By the definition of $\Nf$ we have the following 
		\begin{equation*}
			|\mathfrak{N}_{f, a}(l, 1, 1, 1)-\theta(l)\mathfrak{N}_{f, a}(1, 1, 1, 1)| = \frac{\theta(l)}{q}\left|\sum_{\psi_\ell} \chi_{f,a}(l, 1, 1, 1)\right| , 
		\end{equation*}
		where 
		\begin{equation*}
			\left|\chi_{f,a}(l, 1, 1, 1)\right|=\left|\sum_{u\in\F} \psi_1(-au)\sum_{\alpha\in\Fm\setminus P} \chi_{l}(\alpha)\hat{\psi}_1(u\alpha^{-1})\right| \leq q^{m/2+1}+nq .
		\end{equation*}
		
		Therefore $|\mathfrak{N}_{f, a}(l, 1, 1, 1)-\theta(l)\mathfrak{N}_{f, a}(1, 1, 1, 1)|\leq \theta(l)(q^{m/2}+n)$. Similarly we get, 
		\begin{equation*}
			\left|\chi_{f,a}(1, l, 1, 1)\right|=\left|\sum_{u\in\F} \psi_1(-au)\sum_{\alpha\in\Fm\setminus P} \chi_{l}(f(\alpha))\hat{\psi}_1(u\alpha^{-1})\right| \leq (n+	1)q^{m/2+1}
		\end{equation*}
		
		This implies that $|\mathfrak{N}_{f, a}(1, l, 1, 1)-\theta(l)\mathfrak{N}_{f, a}(1, 1, 1, 1)|\leq \theta(l)(n+	1)q^{m/2}$.
		
		Again, for an irreducible divisor $h$ of $x^m-1$, 
		\begin{align*}
			\left|\chi_{f,a}(1, 1, h, 1)\right|=& \left|\sum_{u\in\F} \psi_1(-au)\sum_{\alpha\in\Fm\setminus P} \psi_{h}(\alpha)\hat{\psi}_1(u\alpha^{-1})\right| \\ 
			=& \left|\sum_{u\in\F} \psi_1(-au)\sum_{\alpha\in\Fm\setminus P} \hat{\psi}_1(u_1\alpha +u\alpha^{-1})\right| \\ \leq& ~2q^{m/2+1}+nq.
		\end{align*}
		Therefore,  $|\mathfrak{N}_{f, a}(1, 1, h, 1)-\Theta(h)\mathfrak{N}_{f, a}(1, 1, 1, 1)|\leq \Theta(h)(2q^{m/2}+n)$.
		
		Again, similarly we have 
		\begin{align*}
			\left|\chi_{f,a}(1, 1, 1, h)\right|=& \left|\sum_{u\in\F} \psi_1(-au)\sum_{\alpha\in\Fm\setminus P} \psi_{h}(f(\alpha))\hat{\psi}_1(u\alpha^{-1})\right| \\ 
			=& \left|\sum_{u\in\F} \psi_1(-au)\sum_{\alpha\in\Fm\setminus P} \hat{\psi}_1(u_2f(\alpha) +u\alpha^{-1})\right| \\ \leq& ~(n+1)q^{m/2+1}.
		\end{align*}
		We then have $|\mathfrak{N}_{f, a}(1, 1, 1, h)-\Theta(h)\mathfrak{N}_{f, a}(1, 1, 1, 1)|\leq \Theta(h)(n+1)q^{m/2}$.
		
		Using these bounds in (\ref{cond5.4}) we get,
		\begin{align*}
			\mathfrak{N}_{f, a}(L, L, H, H)-\mathfrak{N}_{f, a}(1, 1, 1, 1)  \geq & -\sum_{i=1}^t\theta(l_i)(q^{m/2}+n) -\sum_{i=1}^t\theta(l_i)(n+1)q^{m/2}  \nonumber \\ & -\sum_{i=1}^u\Theta(h_i)(2q^{m/2}+n) -\sum_{i=1}^u\Theta(h_i)(n+1)q^{m/2} \nonumber \\ & -(2\epsilon_1+2\epsilon_2)\mathfrak{N}_{f, a}(1, 1, 1, 1).
		\end{align*}
		Now, for every $a \in \F$ we have $\mathfrak{N}_{f, a}(1, 1, 1, 1)\leq q^{m-1}$. Also $\sum_{i=1}^t\theta(l_i)=t-\epsilon_1$ and $\sum_{i=1}^u\Theta(h_i)=u-\epsilon_2$, which imply
		\begin{align*}
			\mathfrak{N}_{f, a}(L, L, H, H)-\mathfrak{N}_{f, a}(1, 1, 1, 1) \geq& -(q^{m/2}+n)(t-\epsilon_1) -(n+1)q^{m/2}(t-\epsilon_1)   \\ & -(2q^{m/2}+n)(u-\epsilon_2) -(n+1)q^{m/2}(u-\epsilon_2) \\ &- (2\epsilon_1+2\epsilon_2)\mathfrak{N}_{f, a}(1, 1, 1, 1).
		\end{align*}
		Hence we have 
		\begin{align}\label{cond5.5}
			\mathfrak{N}_{f, a}(L, L, H, H)-\mathfrak{N}_{f, a}(1, 1, 1, 1) \geq& -q^{m/2}\{(n+2)(t-\epsilon_1)+(n+3)(u-\epsilon_2)\} \nonumber \\ & - n(t+u-\epsilon_1-\epsilon_2)-(2\epsilon_1+2\epsilon_2)q^{m-1}.
		\end{align}
		Substituting ~(\ref{cond5.3}) and ~(\ref{cond5.5}) in (\ref{cond5.2}) we get,
		\begin{multline*}
			\mathfrak{N}_{f,a}(q^m-1, q^m-1, x^m-1, x^m-1) \\ \geq \theta(k)^2\Theta(g)^2\left[\lambda.q^{m-1}-q^{m/2}(2n+1)W(k)^2W(g)^2\left(2r+2s-1+2\lambda\right)\right] \\ -q^{m/2}\{(n+2)(t-\epsilon_1)+(n+3)(u-\epsilon_2)\}- n(t+u-\epsilon_1-\epsilon_2)\\-(2\epsilon_1+2\epsilon_2)q^{m-1} \\
			= q^{m/2}\Bigg[(\lambda\theta(k)^2\Theta(g)^2-(2\epsilon_1+2\epsilon_2))q^{m/2-1}-(2n+1)(\theta(k)^2\Theta(g)^2W(k)^2W(g)^2(2r+2s-1+2\lambda)\\ - \frac{(n+2)(t-\epsilon_1)+(n+3)(u-\epsilon_2)}{(2n+1)} - \frac{n}{(2n+1)q^{m/2}}(t+u-(\epsilon_1+\epsilon_2)))\Bigg] .
		\end{multline*}
		Therefore, $\mathfrak{N}_{f,a}(q^m-1, q^m-1, x^m-1, x^m-1)>0$ if 
		\begin{multline}
			q^{m/2-1} > (2n+1)\Bigg[\theta(k)^2\Theta(g)^2W(k)^2W(g)^2(2r+2s-1+2\lambda)+\frac{(n+2)(t-\epsilon_1)+(n+3)(u-\epsilon_2)}{(2n+1)} \nonumber \\ + \frac{n}{(2n+1)q^{m/2}}(t+u-(\epsilon_1+\epsilon_2))\Bigg]/\Bigg[\lambda\theta(k)^2\Theta(g)^2-(2\epsilon_1+2\epsilon_2)\Bigg] . \qedhere
		\end{multline} 
		\end{proof}
		We note that Theorem \ref{thsieve} is a particular case of Theorem \ref{modps}. In order to recover Theorem \ref{thsieve} we just need to set $t=u=\epsilon_1=\epsilon_2=0$. 
	
	\section{Computational results}
	
	In this section we apply our results to analyze the existence of elements with desired properties. The results that are discussed above apply to the arbitrary finite field $\Fm$ of arbitrary characteristic. Due to the complexity of the task and limited computational resources we present the study of the finite fields of characteristic $5$ and rational functions with degree sum $4$. Henceforth, we assume that $q=5^k$, where $k$ is a positive integer and $n=4$. 
	
	From now on, we use the concept of the $5$-free part of $m$ i.e. $m'$, where $m^\prime$ is such that $m= 5^k m^\prime$, such that $\gcd(5,m^\prime)=1$ and $k$ is a non-negative integer. We then have that $\omega(x^m-1)=\omega(x^{m^\prime}-1)$ which implies $W(x^m-1)=W(x^{m^\prime}-1)$. Further, we split our computations in four cases:
	\begin{enumerate}[label=(\roman*)]
		\item $m^\prime\mid q^2-1$,
		\item $m^\prime \nmid q^2-1$,
		\item $m=4$,
		\item $m=3$.
	\end{enumerate}
	
	\subsection{The case $m^\prime \mid  q^2-1$}\label{case(i)}
	
	In this case $x^{m^\prime}-1$ splits into at most a product of $m^{\prime}$ linear factors over $\F$, that is, $\omega(x^{m^\prime}-1)\leq m^\prime$. For further computation, we need some additional results. The following result is inspired from Lemma~6.1 of Cohen's work \cite{11}. 
	
	\begin{lemma}\label{lambdaeq}
		For $q=p^k$, where $k$ is a positive integer, let $d=q^m-1$ and $g\mid x^m-1$ with $g_1, g_2, \ldots, g_s$ be the remaining distinct  irreducible polynomials dividing $x^m-1$. Furthermore, let us write $ \lambda := 1 - 2\sum_{i=1}^s {q^{-deg(g_i)}}$ and  $\Lambda := \frac{2s-1}{\lambda}+2$, with $\lambda>0$. Let $m= m^\prime\, 5^k$, where $k$ is a non-negative integer and $\gcd(m^\prime,5)=1$. If $m^\prime\mid  q-1$, then $$\Lambda = \frac{2q^2+(a-6)q+4}{(a-2)q+2} , $$  where $a= {q-1}/{m^\prime}$. In particular, $\Lambda<2q^2$.
	\end{lemma}
	
	We also need the following lemma which can be derived from Lemma ~6.2 of Cohen's work \cite{11}. We use this result throughout unless stated otherwise.
	
	\begin{lemma}\label{wbd}
		Let $\alpha$ be a positive integer. Then $W(\alpha)<4514.7\alpha^{1/8}$.
	\end{lemma}
	
	Let $d=q^m-1$ and $g=1$ in (\ref{cond4.2}), then we have the sufficient condition $$q^{\frac m2-1} > 9W(q^m-1)^2\Lambda$$ Applying Lemma \ref{wbd} it is sufficient that $q^{\frac m4-3}>18\cdot(4514.7)^2$. Now for $m^\prime=q-1$, the inequality becomes $q^{\frac{q-1}{4}-3}>18\cdot(4514.7)^2$. This holds for $q\geq 125$. 
	
	Now in order to reduce our calculations, we consider the range $28 \leq m^\prime \leq {q-1}/{2}$ for $q \geq 125$. Here we have by Lemma \ref{lambdaeq} that $\Lambda < q^2$. Hence (\ref{cond4.2}) is satisfied if $q^{\frac{m^\prime}{4}-3}> 9\cdot(4514.7)^2$ which holds for $m^\prime \geq 28$. 
	
	When $m^\prime={q-1}/{2}$, then $\Lambda\leq q^2$ and then the inequality is $q^{\frac{m^\prime}{4}-3}> 9\cdot(4514.7)^2$ and this holds for $m^\prime\geq 28$. Since $m^\prime=28 \neq {q-1}/{2}$ for any $q=5^k$, hence we leave this case.
	
	Next,  we  investigate all   cases with $m^\prime<28$. For this, we set $d=q^m-1$, $g=1$ in Theorem~\ref{sieveineq} unless otherwise mentioned.
	
	\subsubsection{The case $m^\prime=1$}
	
	When $m^\prime=1$ we have $m=5^j$ where $j$ is a positive integer. Here we take $g=x-1$ as a consequence we get $\lambda=1$ and $\Lambda=1$. Then, Inequality (\ref{cond4.2}) becomes $$ q^{\frac{5^j}{2}-1}> 9\cdot W(q^m-1)^2.$$ Now, applying Lemma \ref{wbd}, the above inequality becomes $$q^{\frac{5^j}{4}-1}> 9\cdot(4514.7)^2.$$ For $q=5$ this condition holds for all $j \geq 3$. Again for $q=5^2$, this holds for all $j \geq3$; for $5^3 \leq q \leq 5^{47}$, this holds for all $j \geq 2$ and for $q \geq 5^{48}$, this holds for all $j \geq 1$. Thus the condition holds for all pairs $(q,m)$ except these 49 choices.
	
	Now for these pairs we checked the condition $q^{\frac{m}{2}-1}> 9\cdot W(q^m-1)^2 \cdot W(x-1)^2$, that is, $q^{\frac{m}{2}-1}> 36\cdot W(q^m-1)^2$ by directly calculating $W(q^m-1)$. Based on our calculations we get that the pairs $(5, 5), (5^2, 5), (5^3, 5), (5^4, 5), (5^5, 5)$ and $(5^6, 5)$ do not satisfy the condition.
	
	\begin{table}[h]
		\begin{center}\small
			\begin{tabular}{c c c c c c c c c c}
				\hline 
				\# & $(q,m)$ & $d$ & $r$ & $g$ & $s$ & $\lambda>$ & $\Lambda<$ & $q^{m/2-1}$ & $9W(d)^2W(g)^2\Lambda$ \\
				\hline
				1 & $(5, 15)$ & 2 & 5 & $1$ & 2 & 0.2333 & 57.7227 &  34938.5622 & 2078.0165 \\
				
				2 & $(5, 20)$ & 6 & 6 & $x^3+x^2+x+1$ & 1 & 0.1834 & 72.9108 &  $1.95 \times 10^6$ & 671945.3645 \\
				
				3 & $(5, 30)$ & 42 & 8 & $x+1$ & 3 & 0.1164 & 182.3967 &  $6.11 \times 10^9$ & $420242.0506$ \\
				
				4 & $(5^2, 10)$ & 6 & 6 & $x+1$ & 1 & 0.5034 & 27.8281 &  $390625$ & 16028.9471 \\
				
				5 & $(5^2, 12)$ & 546 & 4 & $x+1$ & 11 & 0.0458 & 635.7250 &  $9.765 \times 10^6$ &  $5.859 \times 10^6$ \\
				
				6 & $(5^3, 12)$ & 6 & 9 & $1$ & 8 & 0.2659 & 126.0757 &  $3.05 \times 10^{10}$ & $18154.9109$ \\
				
				7 & $(5^3, 5)$ & 2 & 5 & 1 & 1 & 0.6973 & 17.7752 &  1397.5425 & 639.9063 \\
				
				8 & $(5^3, 6)$ & 6 & 5 & $1$ & 4 & 0.5095 & 35.3693 &  $15625$ & $5093.1744$ \\
				
				9 & $(5^4, 5)$ & 6 & 6 & 1 & 1 & 0.5802& 24.4089 &  15625.00 & 3514.8681 \\
				
				10 & $(5^5, 5)$ & 2 & 6 & 1 & 1 & 0.7566 & 19.1818 &  174692.8108 & 690.5445 \\
				
				11 & $(5^6, 5)$ & 6 & 9 & $x+4$ & 0 & 0.3907 & 45.5122 &  $1.95\times10^6$ & 26214.9767 \\
				
				12 & $(5^5, 6)$ & 6 & 9 & $1$ & 4 & 0.3895 & 66.1988 &  $9.76 \times 10^6$ & $9532.6267$ \\
				
				13 & $(5^7, 6)$ & 6 & 10 & $1$ & 4 & 0.5083 & 55.1086 &  $6.11 \times 10^9$ & $7935.6342$ \\
				
			    14 & $(5^3, 10)$ & 6 & 9 & $x+1$ & 1 & 0.3747 & 52.7078 &  $2.44 \times 10^8$ & 30359.6925 \\
				
				15 & $(5^4, 6)$ & 6 & 6 & $1$ & 6 & 0.4671 & 51.2505 &  $390625$ & 7380.0659  \\
				
				16 & $(5^3, 8)$ & 6 & 6 & $1$ & 6 & 0.4219 & 73.6782 &  $1.95 \times 10^{6}$ & $10609.6639$ \\
				
				17 & $(5^5, 8)$ & 6 & 9 & $1$ & 6 & 0.5661 & 53.2294 &  $3.05 \times 10^{10}$ & $7665.0475$ \\
				
				18 & $(5^6, 6)$ & 6 & 9 & $1$ & 6 & 0.3297 & 89.9560 &  $2.44 \times 10^8$ & 12953.6563 \\
				\hline
			\end{tabular}
			\caption{Pairs $(q, m)$ with $m^\prime \mid  q^2-1$ for which Theorem \ref{thsieve} holds for the above choices of $d$ and $g$.\label{table1}}
		\end{center}
	\end{table}
	
	\subsubsection{The case $m^\prime=2$}
	
	When $m^\prime=2$, we have $m=2\cdot5^j$ where $j$ is a positive integer. In this case we get $\lambda=1-{2}/{q}$ and $\Lambda=2+{3q}/{(q-4)}\leq17$. Then the Inequality (\ref{cond4.2}) becomes $$ q^{\frac{2.5^j}{2}-1}> 9\cdot17\cdot W(q^m-1)^2.$$ Now applying Lemma \ref{wbd}, the above inequality becomes $$q^{\frac{2.5^j}{4}-1}> 9\cdot17\cdot(4514.7)^2.$$ Similarly we have that for $q=5$, this condition holds for all $j \geq 3$, for $5^2 \leq q \leq 5^{20}$, this holds for all $j \geq 2$ and for $q \geq 5^{21}$, this holds for all $j \geq 1$. Proceeding in a similar way and verifying the condition  $q^{\frac{m}{2}-1}> 9\cdot17 \cdot W(q^m-1)^2$ for the above 21 pairs, we get that $(5, 10), (5^2, 10)$ and $(5^3, 10)$ fail to satisfy the condition. 
	
	\subsubsection{The case $m^\prime=3$}
	
	In this case we have $m=3\cdot5^j$ and $q=5^{k}$, where $j, k$ are positive integers. Here we have $\Lambda=2+{5q}/{(q-6)}\leq8.6$. Thus the condition becomes $$ q^{\frac{3.5^j}{2}-1}> 9\cdot(8.6)\cdot W(q^m-1)^2.$$ Again applying Lemma \ref{wbd}, the above inequality becomes $$q^{\frac{3.5^j}{4}-1}> 9\cdot(8.6)\cdot(4514.7)^2.$$ This holds for all pairs except $(5, 15)$, $(5^2, 15)$, $(5^3, 15)$and $(5^4, 15)$. After verifying the condition $q^{\frac{m}{2}-1}> 9\cdot(8.6)\cdot W(q^m-1)^2$ for these four pairs, we find that all except the pair $(5, 15)$ satisfy the condition. Hence, in this case, $(5, 15)$ is the only possible exceptional pair.
	
	\subsubsection{The case $4\leq m'\leq 27$}
	Now we proceed in a similar fashion for the remaining cases $4 \leq m^\prime \leq 27$. We find that for many of these cases, there is no possible exceptional pair. The possible exceptional pairs for the remaining cases $4 \leq m^\prime \leq 27$ are as follows,
	
	$$ (5, 6), (5, 8), (5, 12), (5, 20), (5, 24), (5, 30), (5^2, 6), (5^3, 6), (5^3, 8), $$
	$$ (5^5, 8), (5^3, 12), (5^4, 6), (5^5, 6), (5^6, 6), (5^7, 6), (5^2, 8)  ~\text{and} ~(5^2, 12).$$
	
	Now, for the above possible exceptional pairs we choose suitable values of $d$ and $g$ such that all except the pairs $(5, 5)$, $(5, 6)$, $(5, 8)$, $(5, 12)$, $(5, 24)$, $(5^2, 5)$, $(5, 10)$, $(5^2, 6)$ and $(5^2,8)$  satisfy Theorem~\ref{thsieve} (see Table~\ref{table1}). Therefore, the above nine pairs are the only possible exceptional pairs for the case $m^\prime \mid  q^2-1$.
	
	\subsection{The case $m^\prime \nmid q^2-1$}
	
	Let $e$ be the order of $q$ mod $m^\prime$. Then $x^{m^\prime}-1$ is a product of irreducible polynomial factors of degree less than or equal to $e$ in $\mathbb{F}_q[x]$; in particular, $e\geq 2$ if $m^\prime \nmid q-1$. Let $M$ be the number of distinct  irreducible factors of $x^m-1$ over $\mathbb{F}_q$ of degree less than $e$. Let $\si(q,m)$ denotes the ratio $$ \si(q,m):= \frac{M}{m},$$ where $m\si(q,m)= m^\prime\si(q,m^\prime)$.
	
	We need the following bounds which can be deduced from Proposition~5.3 of \cite{cohen1}.
	
	\begin{lemma}\label{sibd}
		Suppose $q=5^k$. Let $m^\prime>4$ and $m_1=\gcd(q-1, m^\prime)$
		\begin{enumerate}[label=(\roman*)]
			\item For $m^\prime=2m_1$, $\si(q, m^\prime)=1/2$. 
			\item For $m^\prime=4m_1$, $\si(q, m^\prime)=3/8$. 
			\item For $m^\prime=6m_1$, $\si(q, m^\prime)=13/36$. 
			\item Otherwise, $\si(q, m^\prime)\leq 1/3$. 
		\end{enumerate}
	\end{lemma}

	We need the following Lemma 5.3 \cite{rani1} which provides a bound on $\Lambda$, for suitable development of the sufficient condition. 
	
	\begin{lemma}\label{lambd}
		Assume that $q=p^k$ and $m$ is a positive integer such that $m^\prime\nmid q-1$.  Let $e (>2)$ stands for the order of $q \mod m'$.   Let $g$ be the product of the irreducible factors of $x^{m'}-1$ of degree less than $e$.  Then, in the notation of Lemma  \ref{lambdaeq} with $d=q^m-1$, we have $\Lambda \leq 2m^\prime$.
	\end{lemma}
	
	We note that $m^\prime= 1, 2, 3, 4 ~\text{and} ~6$ divide $q^2-1$ for any $q=5^k$ and have been already discussed in Subsection~\ref{case(i)} above. Thus it suffices to discuss $m^\prime \geq 7$ for which $m^\prime \nmid q^2-1$.

	\begin{table}[h]
		\begin{center}\small
			\begin{tabular}{c c c c c c c c c c}
				\hline 
				\# & $(q,m)$ & $d$ & $r$ & $g$ & $s$ & $\lambda>$ & $\Lambda<$ & $q^{m/2-1}$ & $9W(d)^2W(g)^2\Lambda$ \\
				\hline

				1 & $(5, 11)$ & 2 & 1 & $x+4$ & 2 & 0.9987 & 7.0064 &  $1397.5425$ & $1008.9234$ \\
				
				2 & $(5, 13)$ & 2 & 1 & $x+4$ & 3 & 0.9903 & 9.0679 &  $6987.7125$ & $1305.7706$ \\
				
				3 & $(5, 14)$ & 6 & 3 & $x+1$ & 3 & 0.5262 & 22.9037 &  $15625$ & $13192.5516$ \\
				
				4 & $(5, 18)$ & 6 & 5 & $x^2+4$ & 4 & 0.3814 & 46.5667 &  $390625$ & $107289.6537$ \\
				
				5 & $(5, 21)$ & 2 & 4 & $x+4$ & 4 & 0.8497 & 19.6529 &  $4.36 \times 10^6$ & $2830.0241$ \\
				
				6 & $(5, 22)$ & 2 & 5 & $x^2+4$ & 4 & 0.2136 & 81.5928 &  $9.76 \times 10^6$ & $46997.4470$ \\
				
				7 & $(5, 26)$ & 6 & 3 & $x^2+4$ & 6 & 0.9804 & 19.3405 &  $2.44 \times 10^8$ & $44560.4267$ \\
				
				8 & $(5, 28)$ & 6 & 5 & $x^3+x^2+x+1$ & 5 & 0.3721 & 53.0589 &  $1.22 \times 10^9$ & $488990.4206$ \\
				
				9 & $(5, 32)$ & 6 & 6 & $x^3+x^2+x+1$ & 7 & 0.1548 & 163.5391 &  $3.05 \times 10^{10}$ & $1.51 \times 10^6$ \\
				
				10 & $(5, 36)$ & 6 & 9 & $x^3+x^2+x+1$ & 9 & 0.0099 & 3513.7601 &  $7.62 \times 10^{11}$ & $3.23 \times 10^7$ \\
				
				11 & $(5, 42)$ & 6 & 10 & $x^2+4$ & 8 & 0.3477 & 102.6685 &  $9.54 \times 10^{13}$ & $236548.2720$ \\
				
				12 & $(5, 48)$ & 6 & 9 & $x^{48}-1$ & 0 & 0.3683 & 48.1480 &  $1.19 \times 10^{16}$ & $7.62 \times 10^{15}$ \\
				
				13 & $(5, 16)$ & 6 & 4 & $1$ & 8 & 0.3891 & 61.1042 &  $78125$ & $8799.0116$ \\
				
				14 & $(5^2, 7)$ & 2 & 4 & $1$ & 3 & 0.1796 & 74.4012 &  $3125$ & $2678.4426$ \\
				
				15 & $(5^2, 9)$ & 6 & 5 & $1$ & 5 & 0.3015 & 65.0286 &  $78125$ & $9364.1076$ \\
				
				16 & $(5^2, 11)$ & 2 & 5 & $1$ & 3 & 0.1361 & 112.1752 &  $1.95 \times 10^6$ & $4038.3076$ \\
				
				17 & $(5^2, 13)$ & 2 & 4 & $1$ & 7 & 0.2337 & 91.8591 &  $4.88 \times 10^7$ & $3306.9261$ \\
				
				18 & $(5^2, 14)$ & 6 & 5 & $1$ & 6 & 0.6121 & 36.3071 &  $2.44 \times 10^8$ & $5228.2112$ \\
				
				19 & $(5^2, 16)$ & 6 & 6 & $1$ & 12 & 0.06837 & 513.9088 &  $6.10 \times 10^9$ & $74002.8646$ \\
				
				20 & $(5^2, 18)$ & 42 & 8 & $1$ & 10 & 0.1357 & 259.9584 &  $1.52 \times 10^{11}$ & $149736.0317$ \\
				
				21 & $(5^2, 21)$ & 6 & 10 & $1$ & 9 & 0.2676 & 140.2269 &  $1.90 \times 10^{13}$ & $20192.6820$ \\
				
				22 & $(5^2, 36)$ & 546 & 12 & $x^3+x^2+x+1$ & 17 & 0.01522 & 3747.2929 &  $5.82 \times 10^{23}$ & $5.52 \times 10^8$ \\
				
				23 & $(5^3, 7)$ & 2 & 4 & $1$ & 7 & 0.9137 & 24.9830 &  $174692.8107$ & $899.3900$ \\
				
				24 & $(5^3, 9)$ & 2 & 7 & $1$ & 5 & 0.7850 & 31.2985 &  $2.18 \times 10^7$ & $1126.7472$ \\
				
				25 & $(5^4, 7)$ & 2 & 6 & $1$ & 3 & 0.1027 & 167.4261 &  $9.76 \times 10^{6}$ & $6027.3412$ \\
				
				26 & $(5^4, 9)$ & 6 & 9 & $1$ & 5 & 0.3208 & 86.1439 &  $6.10 \times 10^{9}$ & $12404.7359$ \\

				\hline
			\end{tabular}
			\caption{Pairs $(q, m)$ with $m^\prime \nmid q^2-1$ for which Theorem \ref{thsieve} holds for the above choices of $d$ and $g$.\label{table2}}
		\end{center}
	\end{table}
	
	To establish the next result we need the following bound which can be easily derived from Lemma ~6.2 of Cohen's work \cite{11}. 
	
	\begin{lemma}\label{wbd2}
Let $\alpha$ be a positive integer, such that $5\nmid \alpha$. Then $W(\alpha)<(6.52 \cdot 10^8)\alpha^{1/10}$.
	\end{lemma}
	
	Now suppose $m^\prime \geq 7$. Let $d= q^m-1$ and $g$ be the product of the irreducible factors of $x^{m'}-1$ of degree less than $e$. Thus, Theorem \ref{thsieve} along with Lemma \ref{lambd} give us the following sufficient condition $$q^{m/2-1}>18\cdot m^\prime\cdot W(q^m-1)^2\cdot2^{2m^\prime\si(q, m^\prime)},$$ that is, it suffices that, 
	\begin{equation}\label{cond6.1}
		q^{m/2-1}>18\cdot m\cdot W(q^m-1)^2\cdot2^{2m\si(q, m^\prime)} .
	\end{equation}
	
	\begin{lemma}
		Let $q=5$ and $m$ be a positive integer such that $m^\prime\nmid q^2-1$. Then for $m^\prime\geq 7$ we have that, for any $a \in \F$ there exists an element $\alpha \in \Fm$ with $\Tr_{\mathbb{F}_{q^m}/\mathbb{F}_{q}}(\alpha^{-1})=a$, for which $(\alpha, f(\alpha))$ is  a primitive normal pair  in $\mathbb{F}_{q^m}$ over $\mathbb{F}_{q}$ except, possibly the pairs $(5, 7)$ and $(5, 9)$.
	\end{lemma}
	
	\begin{proof}
		We consider the following four cases.
		
		\textbf{Case 1:} When $m^\prime \neq 2m_1, 4m_1, 6m_1$, we have from Lemma \ref{sibd} that $\si(q, m^\prime)\leq 1/3$. Applying Lemma \ref{wbd2}, the condition (\ref{cond6.1}) becomes 
		$$q^{\frac{3m}{10}-1}>18\cdot m\cdot (6.52 \cdot 10^8)^2\cdot2^{2m/3}.$$ This holds for $m \geq 2554$. Now for $m \leq 2553$, we checked the condition $q^{\frac{m}{2}-1}> 18\cdot m \cdot W(q^m-1)^2 \cdot 2^{2m/3}$ by directly factoring $q^m-1$ and found 43 pairs which do not satisfy the condition viz. $m =$ 7, 9, 11, 13, 14, 17, 18, 19, 21, 22, 23, 26, 27, 28, 29, 31, 32, 33, 34, 35, 36, 37, 38, 39, 42, 44, 45, 46, 48, 49, 51, 52, 54, 55, 56, 63, 64, 66, 68, 70, 72, 84 and 90. 
		
		Now for the above values of $m$, we consider $d=q^m-1$ and $g=x^{m^\prime}-1$ and apply the prime sieve to find that 29 pairs satisfy the condition. Next, by choosing appropriate values of $d$ and $g$ we get that the pairs for $m=$ 11, 13, 14, 18, 21, 22, 26, 28, 32, 36, 42 and 48 satisfy Theorem \ref{thsieve}(refer Table \ref{table2}). Hence in this case the possible exceptional pairs are (5, 7) and (5, 9).
		
		\textbf{Case 2:} When $m^\prime = 2m_1$, we have from Lemma \ref{sibd} that $\si(q, m^\prime) = 1/2$. In this case, we note that the only possible value of $m^\prime$ is 8, which is not possible. Hence in this case there is no possible exception.
		
		\textbf{Case 3:} When $m^\prime = 4m_1$, we have from Lemma \ref{sibd} that $\si(q, m^\prime) = 3/8$. Here, we note that the only possible value of $m^\prime$ is 16. Now, $x^{m^\prime}-1$ can be factorized into 16 distinct linear factors. Thus, using Lemma \ref{wbd} and the above fact, the sufficient condition becomes $$q^{m/4-1}>9\cdot (4514.7)^2 \cdot (2^{16})^2.$$ This holds for $m \geq 107$. For the values of $m \leq 106$, after verifying the condition $q^{m/2-1}> 9\cdot(2^{16})^2 \cdot W(q^m-1)^2$ by directly calculating $\omega(q^m-1)$, we get that the only possible exceptional pair to be (5, 16). Here, after choosing suitable values of $d$ and $g$, we get that (5, 16) satisfies Theorem \ref{thsieve}. So, there is no exceptional pair in this case.
		
		\textbf{Case 4:} When $m^\prime = 6m_1$, we have from Lemma \ref{sibd} that $\si(q, m^\prime) = 3/8$. Again, we note that the only possible value of $m^\prime$ is 24, which is not possible. Thus in this case we find no possible exceptional pair.
	\end{proof}
	
	\begin{lemma}
		Let $q \geq 5^2$ and and $m$ be a positive integer such that $m^\prime\nmid q^2-1$. Then for all pairs $(q, m)$ and for any $a \in \F$ there exists an element $\alpha \in \Fm$ with $\Tr_{\mathbb{F}_{q^m}/\mathbb{F}_{q}}(\alpha^{-1})=a$, for which $(\alpha, f(\alpha))$ is  a primitive normal pair  in $\mathbb{F}_{q^m}$ over $\mathbb{F}_{q}$.
	\end{lemma}
	
	\begin{proof}
		We separate our discussion in four cases as follows.
		
		\textbf{Case 1: $m^\prime \neq 2m_1, 4m_1, 6m_1$.} From Lemma \ref{sibd}, $\si(q, m^\prime)\leq 1/3$ and using Lemma \ref{wbd} in (\ref{cond6.1}), we get, $$q^{m/4-1}>18\cdot m\cdot (4524.7)^2\cdot2^{2m/3}.$$ This holds for $q^m= 25$ and $m \geq 80$. Now for $q \geq 25$ and $q^m < 25^{80}$, we checked $$q^{m/2-1}>18 \cdot m \cdot W(q^m-1)^2 \cdot 2^{2m/3}$$ and got the possible exceptional pairs to be $(5^2, 7), (5^2, 9), (5^2, 11), (5^2, 13), (5^2, 14), (5^2, 16),\\ (5^2, 18), (5^2, 21), (5^2, 36), (5^3, 7), (5^3, 9), (5^4, 7), (5^4, 9).$
		
		\textbf{Case 2: $m^\prime = 2m_1$.} Here $\si(q, m^\prime)=1/2$. Applying Lemma \ref{wbd}, the condition (\ref{cond6.1}) becomes $$q^{m/4-1}>18 \cdot m \cdot (4514.7)^2 \cdot 2^m.$$ If $q=25$ for the condition to hold it suffices that $m \geq 256$. Again, on verification of $q^{m/2-1}>18 \cdot m \cdot W(q^m-1)^2 \cdot 2^{m}$, we get that there is no exceptional pair.
		
		\textbf{Case 3: $m^\prime = 4m_1$.} In this case $\si(q, m^\prime)=3/8$. Similarly, the condition (\ref{cond6.1}) becomes $$q^{m/4-1}>18 \cdot m \cdot (4514.7)^2 \cdot 2^{3m/4},$$ which holds for $q= 25$, for all $m \geq 97$. For $q \geq 25$ and $q^m \leq 25^{94}$ upon verification we get that $q^{m/2-1}>18 \cdot m \cdot W(q^m-1)^2 \cdot 2^{3m/4}$ holds for all pairs. Thus, in this case there is no exceptional pair.
		
		\textbf{Case 4: $m^\prime = 6m_1$.} In this case $\si(q, m^\prime)=13/36$. Here, the condition (\ref{cond6.1}) becomes $$q^{m/4-1}>18 \cdot m \cdot (4514.7)^2 \cdot 2^{13m/18},$$ which holds for $q= 25$, for all $m \geq 91$. Similarly, for $q \geq 25$ and $q^m \leq 25^{90}$, upon further verification we get that there is no exceptional pair.
		
		Finally, we refer to Table \ref{table2}, to note that Theorem \ref{thsieve} holds for all the possible exceptional pairs in the above discussion for the suitable choices of $d$ and $g$. Thus, there is no possible exception in this case.
	\end{proof}
	
	Next, for all the above possible exceptional pairs we verified Theorem~\ref{modps} and found that for none of the pairs the modified prime sieving criterion (\ref{cond5.1}) holds for values of $k, L, g$ and $H$ chosen suitably. Thus, compiling all the computational results of this section, we conclude the following result.
	
	\begin{theorem}\label{main_exact_theorem}
		Let $q=5^k$ and $m \geq 5$. Then for any $a \in \F$, there exists an element $\alpha \in \Fm$ with $\Tr_{\mathbb{F}_{q^m}/\mathbb{F}_{q}}(\alpha^{-1})=a$, for which $(\alpha, f(\alpha))$ is  a primitive normal pair  in $\mathbb{F}_{q^m}$ over $\mathbb{F}_{q}$ unless $(q, m)$ is one of the pairs $(5, 5)$, $(5^2, 5)$, $(5, 10)$, $(5^2, 6)$, $(5^2,8)$, $(5, 6)$, $(5, 7), (5, 9), (5, 8), (5, 12)$ and $(5, 24)$
	\end{theorem}
	\subsection{The case $m=4$}\label{m=4}
	In this case, $x^4-1$ splits into four linear factors over any field of characteristic $5$. Thus, $W(x^4-1)=2^4$. It follows from Theorem~\ref{main_theorem} that $(5^k,4)\in\T$ if
	\begin{equation} \label{m=4_cond}
	q > 9\cdot W(q^4-1)^2 \cdot 2^8,
	\end{equation}
	where $q=5^k$. To proceed further, we need the following bounds that can be easily derived from Lemma~6.2 of Cohen's work \cite{11}. 
	
	\begin{lemma}\label{wbd3}
         Let $\alpha$ be a positive integer, such that $5\nmid \alpha$. Then
         \begin{enumerate}[label=(\roman*)]
         	\item $W(\alpha)<2760.35\cdot \alpha^{1/8}$. 
         	\item $W(\alpha)<299243.62\cdot \alpha^{1/9}$.
         	\item $W(\alpha)<2.8592\cdot 10^{67}\cdot \alpha^{1/14}$.
         \end{enumerate}
	\end{lemma}
	
	Using the above, we replace $W(q^4-1)$ by $299243.62\cdot q^{4/9}$ in the condition of \eqref{m=4_cond}. Thus, the case $q>6.77\cdot 10^{128}$ is settled.
	
	The fact that this number is exceptionally large pushes us to adopt the following reduction technique, see \cite{cohengupta21,cohenkapetanakis20,cohenkapetanakis20b}. Let $p_i$ be the $i$-th prime, that is, $p_1=2$, $p_2=3$, $p_3=5$ and so on. Further, let $t_{m,k}$ be the number of distinct prime divisors of $5^{mk}-1$. Clearly, if $q=5^k$,
	\[ q^m-1 \geq \prod_{i=1}^{t_{m,k}} p_i \ \Rightarrow q > \left( \prod_{i=1}^{t_{m,k}} p_i\right)^{1/m} . \]
	A quick computation reveals that $\left( \prod\limits_{i=1}^{202} p_i\right)^{1/4} > 6.77\cdot 10^{128}$, that is,  the case $t_{4,k}\geq 202$ is settled. Then, we observe that if, for some $k$, there exist some $(t_1,t_2)\in\mathbb Z^2$, such that $t_1\leq t_{4,k}\leq t_2$ and
	\begin{equation} \label{eq:al1} \left( \prod_{i=1}^{t_1} p_i \right)^{1/4} > 9\cdot 4^{t_2} \cdot 4^4 , \end{equation}
	then Theorem~\ref{main_theorem} implies that the case $q=5^k$ is settled. We verify the validity of the above condition for $(t_1,t_2)$ equal to $(196,201)$, $(191,195)$, $(190,187)$, $(186,184)$, $(181,183)$, $(179,180)$, $(177,178)$, $(176,176)$, $(175,175)$, $(174,174)$ and $(173,173)$. So far, the case $t_{4,k}\geq 173$ is settled.
	
	However, the above reduction is inadequate for a concrete result, so we proceed by pushing the method even further closer to its limits. More precisely, we combine the above with prime sieve technique.
	Notice that, if the prime factors of $5^{4k}-1$ are (in ascending order) $q_1,\ldots ,q_{t_{4,k}}$, then, for every $i$ we have that $q_i\geq p_i$. It follows from Theorem~\ref{thsieve}, that the condition of \eqref{eq:al1} can be reduced to
	\begin{equation} \label{eq:al2}
	\left( \prod_{i=1}^{t_1} p_i \right)^{1/4} > 9\cdot 4^{t_2-r} \cdot \Lambda,
	\end{equation}
	where $r$ can be any integer within the range $0\leq r\leq t_1$ and $\Lambda = (2r+7)/\lambda+2$, where \[ \lambda = 1-\sum_{i=1}^{r} \frac{2}{p_i}-\frac{8}{\left( \prod\limits_{i=1}^{t_1} p_i \right)^{1/4}}. \]
	Note that we chose to sieve all the linear factors of $x^4-1$. In this application, we also choose $r=\lfloor t_1/2\rfloor$ and we easily verify the validity of \eqref{eq:al2} for the cases, where $(t_1,t_2)$ equals to $(122,172)$, $(92,121)$, $(72,91)$, $(59,71)$, $(50,58)$, $(44,49)$, $(40,43)$, $(38,39)$, $(36,37)$, $(34,35)$, $(33,33)$ and $(32,32)$. It follows that the case $t_{4,k}\geq 32$ is settled, which in turn implies that the case $q\geq 1.4155\cdot 10^{12}$ is settled, that is, the case $q=5^k$ with $k>17$ is settled.
	
	We proceed by checking the condition in \eqref{m=4_cond} by explicitly computing $W(q^4-1)$ for all $q=5^k$ and $1\leq k\leq 17$. It turns out that this is satisfied for all values of $k$, besides $k=1,2,3,4,5,6,7,8,9,10,12,15$. Finally, we successfully apply the prime sieve technique, see Theorem~\ref{thsieve}, for $k=6,7,8,9,10,12,15$, where we sieve the entire additive part and the maximum number of prime factors of $q^4-1$, such that $\lambda$ remains positive.
	
	Summing up, we managed to prove the desired result, with theoretical means, for the case $m=4$ and $q=5^k$, for all $k\neq 1,2,3,4,5$.
	\subsection{The case $m=3$}\label{m=3}
	Finally, we focus on the $m=3$ case. Here, over $\mathbb{F}_{5^k}$, the polynomial $x^3-1$ splits into three linear factors, if $k$ is even, and into one linear and one quadratic factor, if $k$ is odd. It follows, from Theorem~\ref{main_theorem}, that $(5^k,3)\in\T$ if
	\begin{equation} \label{m=3_cond}
	5^{k/2} > \begin{cases} 9\cdot W(q^3-1)^2 \cdot 2^6, &\text{if $k$ is even}, \\ 9\cdot W(q^3-1)^2 \cdot 2^4, &\text{if $k$ is odd}. \end{cases}
	\end{equation}
	
	Using Lemma~\ref{wbd3}, we replace $W(q^3-1)$ by $2.8592\cdot 10^{67}\cdot q^{1/14}$ in the condition of \eqref{m=3_cond} and first consider the strongest condition of the two. This yields that the case $q>2.63471\cdot 10^{1927}$ is settled. This implies that the case $q=5^k$, for $k\geq 2758$, is settled.
	
	We proceed with the same strategy as in the $m=4$ case, but this time we move immediately to the sieving part. Namely, the equivalent to \eqref{eq:al2} condition to our case is
	\begin{equation}\label{eq:al3}
	\left( \prod_{i=1}^{t_1} p_i \right)^{1/6} > 9\cdot 4^{t_2-r} \cdot \Lambda,
	\end{equation}
	where $r$ can be any integer within the range $0\leq r\leq t_1$ and $\Lambda = (2r+5)/\lambda+2$, where 
	\[ \lambda = 1-\sum_{i=1}^r \frac{2}{p_i}-\frac{6}{\left( \prod\limits_{i=1}^{t_1} p_i \right)^{1/3}}. \]
	Just like in the $m=4$ case, we chose to sieve all the linear factors of $x^3-1$, but, in this occasion, we also choose $r=\lfloor 3t_1/4\rfloor$. We verify the validity of \eqref{eq:al3} for the cases when $(t_1,t_2)$ equals to $(1578,2757)$, $(943,1577)$, $(588,942)$, $(383,587)$, $(260,382)$, $(184,259)$, $(136,183)$, $(104,135)$, $(83,103)$, $(68,82)$, $(58,67)$, $(51,57)$, $(46,50)$, $(42,45)$, $(39,41)$, $(37,38)$, $(36,36)$, $(35,35)$ and $(34,34)$.
	It follows that the case $t_{4,k}\geq 34$ is settled, which in turn implies that the case $q\geq 4.1611\cdot 10^{17}$ is settled, that is, the case $q=5^k$ with $k>25$ is settled.
	
	We proceed by checking the condition in \eqref{m=3_cond} with explicitly computing $W(q^3-1)$ for all $q=5^k$ and $1\leq k\leq 25$ and using the appropriate condition of the two, depending on the parity of $k$. A quick computation reveals that this is satisfied for all values of $k\leq 25$, except 
	\[ k=1, 2, 3, 4, 5, 6, 7, 8, 9, 10, 11, 12, 13, 14, 15, 16, 17, 18, 20, 21, 22, 24, 25. \]
	Among the above, we successfully apply the prime sieve technique, see Theorem~\ref{thsieve}, for $k=9$, and $11\leq k\leq 25$, where, once again, we sieve the entire additive part and the maximum number of prime factors of $q^4-1$, such that $\lambda$ remains positive.
	
	Summing up, we managed to prove the desired result, with theoretical means, for the case $m=3$ and $q=5^k$, for all $k\neq 1,2,3,4,5,6,7,8,10$.

	Again for all the above possible exceptional pairs we verified Theorem~\ref{modps} and found that for none of the pairs the modified prime sieving criterion (\ref{cond5.1}) holds for values of $k, L, g$ and $H$ chosen suitably. 
	\subsection{Concluding remarks}
	The results of Subsections~\ref{m=4} and \ref{m=3}, in conjunction with Theorem~\ref{main_exact_theorem}, yield the following:
	\begin{theorem}\label{exact_theorem_comb}
			Let $q=5^k$, $m \geq 3$ and $f\in\Fm(x)$ of degree sum $4$. Then for any $a \in \F$, there exists an element $\alpha \in \Fm$ with $\Tr_{\mathbb{F}_{q^m}/\mathbb{F}_{q}}(\alpha^{-1})=a$, for which $(\alpha, f(\alpha))$ is  a primitive normal pair  in $\mathbb{F}_{q^m}$ over $\mathbb{F}_{q}$ unless $(q, m)$ is one of the pairs $(5, 5)$, $(5^2, 5)$, $(5, 10)$, $(5^2, 6)$, $(5^2,8)$, $(5, 6)$, $(5, 7)$, $(5, 9)$, $(5, 8)$, $(5, 12)$, $(5, 24)$, $(5,4)$, $(5^2,4)$, $(5^3,4)$, $(5^4,4)$, $(5^5,4)$, $(5,3)$, $(5^2,3)$, $(5^3,3)$, $(5^4,3)$, $(5^5,3)$, $(5^6,3)$, $(5^7,3)$, $(5^8,3)$ or $(5^{10},3)$.
	\end{theorem}
	We conclude this work with an attempt to completely resolve the case $q=5^k$, $m \geq 3$ and $f\in\Fm(x)$ of degree sum $4$, i.e., the case of interest of Theorem~\ref{exact_theorem_comb}. This means that the \emph{possible} exceptions appearing in its statement have to be examined one-by-one. However the vast number of choices of $f$ render this impossible within a reasonable time limit.
	
	Instead, we run a computer script that checks the existence of some $\alpha_a \in \Fm$ with $\Tr_{\mathbb{F}_{q^m}/\mathbb{F}_{q}}(\alpha_a^{-1})=a$, for all $a\in\mathbb{F}_q$, for which $(\alpha, f(\alpha))$ is  a primitive normal pair  in $\mathbb{F}_{q^m}$ over $\mathbb{F}_{q}$, where $f=f_1/f_2\in \E_4$ and $f_1,f_2\in\Fm[x]$ are monic of positive degrees. The basic structure of our script is the following (given $(q,m)$ and some $f\in\E_4$):
	\begin{enumerate}[label=(\roman*)]
	  \item Take a primitive element $\alpha\in\Fm$ and the set $T = \emptyset$
	  \item For $1\leq j\leq q^m-1$, if $\gcd(j,q^m-1)=1$, check if $f(\alpha^j)$ is primitive and normal and if $\alpha^j$ is normal. \label{step:j}
	  \item If yes, calculate $\Tr (\alpha^{-j})$ and add it to $T$.
	  \item If at any point $|T|=q$, return {\tt True}. \label{step:T}
	  \item If $j$ reaches $q^m-1$ and still $|T|<q$, return {\tt False}.
	\end{enumerate}
	In particular, for all pairs $(q,m)$, that appear as possible exceptions on Theorem~\ref{exact_theorem_comb}, we checked for a number of appropriate choices of $f$ with the above restrictions, with this number for most pairs being close to 60000. However this number had to be reduced significantly for pairs where the base field was very large, or even reach zero in two cases. More precisely, the number of rational functions $f\in\E_4$ checked per pair $(q,m)$ is presented in Table~\ref{tab:checked_functions}.
	
	\begin{table}[h]
	\begin{center}
	  \begin{tabular}{ccc|ccc|ccc}
	    \hline \# & $(q,m)$ & $\# f\in\E_4$ & \# & $(q,m)$ & $\# f\in\E_4$ & \# & $(q,m)$ & $\# f\in\E_4$ \\
	    \hline
	    1 & $(5,5)$ & $60746$ & 10 & $(5,12)$ & $60746$ & 19 & $(5^3,3)$ & $60746$ \\
	    2 & $(5^2,5)$ & $60746$ & 11 & $(5,24)$ & $60746$ & 20 & $(5^4,3)$ & $2838$ \\
	    3 & $(5,10)$  & $60746$ & 12 & $(5,4)$ & $3$ & 21 & $(5^5,3)$ & $594$ \\
	    4 & $(5^2,6)$ & $60746$ & 13 & $(5^2,4)$ & $60746$ & 22 & $(5^6,3)$ & $178$ \\
	    5 & $(5^2,8)$ & $60746$ & 14 & $(5^3,4)$ & $60746$ & 23 & $(5^7,3)$ & $19$ \\
	    6 & $(5,6)$ & $60746$ & 15 & $(5^4,4)$ & $1480$ & 24 & $(5^8,3)$ & $0$ \\
	    7 & $(5,7)$ & $60746$ & 16 & $(5^5,4)$ & $594$ & 25 & $(5^{10},3)$ & $0$ \\
	    8 & $(5,9)$ & $60746$ & 17 & $(5,3)$ & $4$ \\
	    9 & $(5,8)$ & $60746$ & 18 & $(5^2,3)$ & $60746$ \\
	    \hline
	  \end{tabular}
	\end{center}
	  \caption{Number of functions $f\in\E_4$ checked per pair $(q,m)$.\label{tab:checked_functions}}
	\end{table}
	A few remarks regarding Table~\ref{tab:checked_functions} are the following:
	\begin{enumerate}[label=(\roman*)]
	  \item The pairs $(5,4)$ and $(5,3)$ turned out to be the only verified genuine exceptions. More precisely, for the pair $(5,4)$ we checked three $f\in\E_4$, namely, $x^2/(x^2+1)$, $x/(x^3+1)$ and $x^3/(x+1)$ and all of them failed. For the pair $(5,3)$ we checked in addition the function $x^2/(x^2+2)$ which also failed, but in this occasion the function $x^2/(x^2+1)$ was not an exception, that is, 3 out of 4 functions failed.
	  \item As already mentioned, when the base field's cardinality was $q\geq 5^4$, the script started being slow. The underlying reason is that, not only the extension field grows at the same time, making every calculation more expensive in computer time, but the fact that the set $T$ will have to grow exponentially more (see Step~\ref{step:T} above), thus Step~\ref{step:j} is repeated exponentially more times. In particular, in the cases $q=5^8$ and $q=5^{10}$ the script was running several days but still did not verify a single $f\in\E_4$. However, the large cardinality of $\Fm$ suggest that these cases are not genuine exceptions.
	\end{enumerate}
	
	All in all, this search needed several weeks of computer time and it turns out that the only verified exceptions are $(q,m)$ being $(5,4)$ and $(5,3)$ and, as already mentioned, we are confident that these are in fact the only exception. So, even though an exhaustive search is out of reach, at least for our limited computational resources, we have collected strong computational data that, combined with Theorem~\ref{exact_theorem_comb}, are enough to support the following:
	\begin{Conjecture}
	Let $q=5^k$, $m \geq 3$ and $f\in\Fm(x)$ of degree sum $4$. Then for any $a \in \F$, there exists an element $\alpha \in \Fm$ with $\Tr_{\mathbb{F}_{q^m}/\mathbb{F}_{q}}(\alpha^{-1})=a$, for which $(\alpha, f(\alpha))$ is  a primitive normal pair  in $\mathbb{F}_{q^m}$ over $\mathbb{F}_{q}$ unless $(q, m)$ is one of the pairs $(5,4)$ or $(5,3)$.
	\end{Conjecture}

	\section{Acknowledgements}
	
	The first author is supported by DST INSPIRE Fellowship(IF210206).
	\section{Declarations}
	
	\textbf{Conflict of interest} The authors declare that there is no conflict of interest.
	
	\bibliographystyle{plain}
	\bibliography{reference}
	
\end{document}